\documentclass[12pt]{article}
\usepackage{amsmath,amssymb,amsthm}
\usepackage{graphicx}
\usepackage[utf8x]{inputenc}
  \newtheorem{theorem}{Theorem}
 
 \newtheorem{claim}[theorem]{Claim}
 
 \newtheorem{proposition}[theorem]{Proposition}
 
 {\theoremstyle{definition}
 \newtheorem{definition}[theorem]{Definition}
 \newtheorem{remark}[theorem]{Remark}
 \newtheorem{notation}[theorem]{Notation}

 }
\newcommand{\N}{\ensuremath{\mathbb N}} 
\newcommand{\R}{\ensuremath{\mathbb R}} 
\newcommand{\Z}{\ensuremath{\mathbb Z}} 
\usepackage{buczo}
\usepackage{tempcrm}

\renewcommand{\lll}{\lambda}
\title{Upper Minkowski dimension estimates for convex restrictions}
\author{Zolt\'an Buczolich\thanks{
Research supported by National Research, Development and Innovation Office--NKFIH, Grant 104178.
\newline\indent {\it 2000 Mathematics Subject
Classification:} Primary : 26A15; Secondary : 26A12, 26A51, 28A78.
\newline\indent {\it Keywords:} Minkowski dimension, H\"older class, restriction, convexity.},
Department of Analysis, E\"otv\"os Lor\'and\\
University, P\'azm\'any P\'eter S\'et\'any 1/c, 1117 Budapest, Hungary\\
email: buczo@cs.elte.hu\\
{\tt www.cs.elte.hu/\hbox{$\sim$}buczo}
}
\date{\today}
\begin{document}
\maketitle

\medskip


\begin{abstract}
We show that there are functions  $f$ in the Hölder class $C^{\aaa}[0,1]$,
$1< \aaa<2$
  such that 
  $f|_{A}$ is not convex, nor concave for any $A\sse [0,1]$ with $\udimm A>\aaa-1$.
  
Our earlier result shows that for the typical/generic $f\in\cea$,
$0\leq \aaa<2$   there is always
 a set $A\sse [0,1]$  such that $f|_A$ is convex and $\udimm A=1$.

 The analogous statement for monotone restrictions
is the following: there are functions  $f$ in the Hölder class $C^{\aaa}[0,1]$,
$1/2 \leq \aaa<1$
  such that 
  $f|_{A}$ is not monotone on $A\sse [0,1]$ with $\udimm A>\aaa$.
This statement is not true for the range of parameters $\aaa<1/2$
and the main theorem of this paper for the parameter range
$1< \aaa <3/2$ cannot be obtained by integration of the result about monotone
restrictions.
  \end{abstract}


\section{Introduction}

In an earlier paper \cite{[crc]} we discussed results about
convex and monotone restrictions of functions belonging to different Hölder
classes. This paper was related to \cite{[ABMP]}, \cite{[KK1]},
\cite{[KK2]}, \cite{[Ma1]} and several others. In  \cite{[ABMP]} and \cite{[crc]} one can read more about the background of these results.

 We denote by $\dimh A$, $\udimm A$, and $\ldimm A$  the Hausdorff, the lower and upper Minkowski (box) dimension of the set $A$, respectively.
In \cite{[crc]} we showed that for the generic/typical $f\in \cea$, 
when $0\leq \aaa<1$ if $A\sse [0,1]$ and $f|_{A}$ is monotone then
 $\ldimm A\leq \aaa$.
 It is rather easy to see that  for any
 $0<\aaa\leq 1$ if $f\in C^{\aaa}[0,1]$ then  there exists 
$A\sse [0,1]$  such that $f|_{A}$ is monotone and $\dimh A \geq \aaa$.
 According to a result from \cite{[ABMP]} 
 if $0<\aaa<1$
and  $B(t)$  is a fractional Brownian motion of Hurst index
$\aaa$ then almost surely  $B|_{A}$ is not monotone increasing for any $A$ with $\udimm A>\max\{ 1-\aaa,\aaa \}$.
Fractional Brownian motion of Hurst index
$\aaa$ belongs to $C^{\aaa -}[0,1]$.
In \cite{[ABMP]} for a dense set of $\widehat{\aaa}$s in $[1/2,1]$ examples of self similar functions  $f\in C^{\widehat{\aaa}}[0,1]$ were also provided
for which  $f|_{A}$ is not monotone for any $A$ with $\udimm A>\widehat{\aaa} $.
I learned from R. Balka about an unpublished argument of A. Máthé, which  implies that for any
function $f:[0,1]\to \R$ one can always find a set $A$ such that 
$f|_{A}$ is monotone and $\udimm A\geq 1/2 $.
 
With respect to convex restrictions in \cite{[crc]} we proved that
 for typical/generic $f\in\cea$ (in the sense of Baire category),
$0\leq \aaa<2$   there is always
 a set $A\sse [0,1]$  such that $f|_A$ is convex and $\udimm A=1$.
 For generic $f\in C_{1}^{\aaa}[0,1]$, $0<\aaa<2$
we have shown that 
  for any $A\sse [0,1]$
 such that  $f|_{A}$ is convex, or concave we have 
$\ldimm A\leq \max \{ 0, \aaa-1 \}.$
In this paper we prove that for $1< \aaa<2$
 there are functions $f\in C^{\aaa}[0,1]$,
  such that 
  $f|_{A}$ is not convex, nor concave for any $A\sse [0,1]$ with $\udimm A>\aaa-1$.
By multiplying with a suitable constant one can obtain functions
$f\in C^{\aaa}_{1}[0,1]$  in our example.
In \cite{[crc]} we also proved for $1<\aaa\leq 2$ for any  $f\in C^{\aaa}[0,1]$
  there is always a set $A\sse[0,1]$  such that $\aaa-1=\dimh A\leq \ldimm A\leq \udimm A$ and $f|_{A}$ is convex, or concave on $A$.
  This shows that the main result of our current paper is best possible.
  
It is interesting that  for $3/2\leq \aaa<2$ 
by integrating Fractional Brownian motions of Hurst index
$\aaa-1$ one can obtain functions $f \in C^{\aaa-}[0,1]$
with the property that $f|_{A}$ is not convex, nor concave for any $A\sse [0,1]$ with $\udimm A>\aaa-1$.
By using from \cite{[ABMP]}  the earlier mentioned dense set of $\widehat{\aaa}$s in $[1/2,1]$
taking integrals of the 
corresponding self-similar functions one  can obtain a dense set in $3/2\leq \aaa<2$ and functions $f \in C^{\aaa}[0,1]$
with the property that $f|_{A}$ is not convex, nor concave for any $A\sse [0,1]$ with $\udimm A>\aaa-1$.
As it was requested by readers of ealier versions of our paper in Proposition
\ref{*integr*} we provide some details of this procedure of obtaining convexity
results from monotonicity results by integration.
 
 For the parameter range $1\leq \aaa <3/2$ it is not possible to take integrals
 of functions constructed for monotone restrictions. Fortunately,
 our proof works for the whole parameter range $(1,2)$. Based on similar ideas 
 the result concerning the case $\aaa=1$ can also be established,
 though we do not make the lengthy and complicated proof of our paper
 even longer and more complicated by considering this case as well.

Since the proof is quite complicated and contains quite a few technical details here we want to give the main heuristic idea of our argument.
\begin{figure}[h]
  \begin{center}
  \includegraphics[width=12cm]{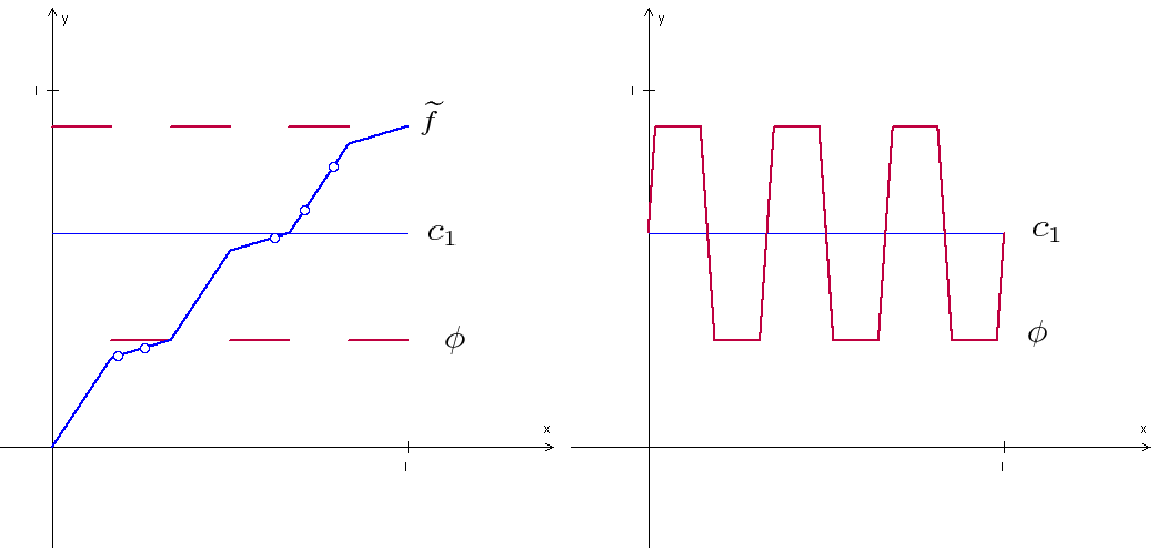}
    \caption{$\fff$, $\tf$ and  $c_{1}$}
     \label{fig1}
  \end{center}
\end{figure}
 Suppose $N$ is a fixed sufficiently large even integer.
 One can consider the function $\fff(x)=c_{1}+c_{2}(-1)^{j+1}$
for $x\in [(j-1)/N,j/N]$, $j=1,...,N$. We take its antiderivative
$\tf(x)=\int_{0}^{x}\fff(t)dt$. These functions are illustrated on the
left half of Figure \ref{fig1}. It is clear that if $\tf$ is convex on a set $A$
then there are at most two $j$'s for which $[(j-1)/N,j/N]$ contains more than 
two points of $A$. Though, it may happen that some other intervals contain
single points of $A$. See  again the left half of Figure \ref{fig1} where
the hollow dots illustrate the correspondig points 
$(a,\tf(a)),$ $a\in A$
on the graph of $\tf$.
We will have to consider in our paper functions  $f$ which
are sufficiently close to $\tf$. These functions will inherit the property that
if $f$ is convex on a set $A$
then there are at most two $j$'s for which $[(j-1)/N,j/N]$ contains more than 
two points of $A$. In the proof of our theorem this will be behind  Claims
\ref{*claimh31*} and \ref{*claimh31b*}.
The function $\fff$ on the left half of Figure \ref{fig1} is not continuous
and we need H\"older $\aaa$ derivative for our function $f$ in Theorem \ref{*5b*}.
Instead of the discontinuous $\fff$ we will use functions $\fff$ which are pictured on
the right half of Figure \ref{fig1} and appear in Subsection \ref{*subsfffn}, especially in \eqref{*h6*b}.  Unfortunately, for these functions
there are intervals, where they are not locally constant. We will call these intervals 
 transitional intervals. We also need to deal with the problem that for a bound
 on the upper box dimension we have to consider all scaling levels
 with grid intervals of the form $[(j-1)/N^{k},j/N^{k}]$, $k\in\N$, $j\in
 \{ 1,...,N^{k} \}$
 and verify that for all sufficiently large $k$ if $f$ is convex on $A$ 
 then we do not have too many such grid intervals which contain points of $A$.
 This means that by subsequent perturbation of the previous functions
 we define the sequence of functions $\fff_{n}$ which will uniformly
 converge to a H\"older $\aaa$ function $g_{1}$. This function $g_{1}$
 is good for our purposes at points where it is ``almost locally constant".
 However due to the existence of transitional intervals we need a second infinite sequence of modifications yielding the sequence $g_{n}$. These functions
 $g_{n}$ will converge to a function $g$ and our function $f$ will be the antiderivative of $g$.

 The author thanks R. Balka and the unknown referee for the careful reading
 of the paper and for comments which
 improved the presentation of the results.

\section{Notation and preliminary results}\label{*secnota*}

For $\aaa\geq 0$ if $f:[0,1]\to\R$ is $\lf \aaa \rf$-times differentiable 
(by definition $f^{(0)}=f$)
we put 
\begin{equation}\label{*defcalla*}
\calla(f)=\sup\Big \{ \frac{|f^{(\lf \aaa \rf)}(x)-f^{(\lf \aaa \rf)}(y)|}{|x-y|^{\{ \aaa \}}}:x\not=y,\  x,y\in [0,1] \Big \}.
\end{equation}
\\
By $C^{0}[0,1]$, or $C[0,1]$ we denote the class of continuous
functions on $[0,1]$.

If $0<\aaa<1$ then $f$ is in $C^{\aaa}[0,1]$, if $\calla(f)<\oo$.

The function $f$ is in $C^{1}[0,1]$ if
$f^{'}$ is continuous on $[0,1]$.
 
If $1<\aaa <2$ then  $f$ is in $ C^{\aaa}[0,1]$ if $f$ is  differentiable, 
and $f^{'}\in C^{\aaa-1}[0,1]$,
that is $\calla(f)<\oo$. 

We denote by $ C^{\aaa-}[0,1]$ 
the set of those functions which are in $C^{\bbb}[0,1]$ for all $\bbb<\aaa$.


For $0\leq \aaa$, $\aaa\not\in \N$,  $f$ is in $C_{1}^{\aaa}[0,1]$ if
$\calla(f)\leq 1$, that is $|f^{(\lf \aaa  \rf)}(x)-f^{(\lf \aaa  \rf)}(y)|\leq |x-y|^{\{\aaa\}}$ for all $x,y\in [0,1]$.



Given an integer $N\geq 2$ and a set $A\sse [0,1]$ we put for $k\in\Z$
$$\can_{N,k}(A)=\#\Big \{ j\in \Z:A\cap \Big [ \frac{j-1}{N^k},\frac{j}{N^k} \Big ]\not=\ess \Big \}.$$

The upper and lower Minkowski (or box) dimension of $A$ is defined as
\begin{equation}\label{*minkdi}
\udimm A=\limsup_{k\to\oo}\frac{\log \can_{N,k}(A)}{k\log N}\text{ and }
\ldimm A=\liminf_{k\to\oo}\frac{\log \can_{N,k}(A)}{k\log N}.
\end{equation}
It is well-known that for any $N$ we obtain the same value.

\begin{notation}
For sets $A\sse [0,1]$ we put 
$${\caf}_{N,k}(A)=\cup\{ [(j-1)/N^{k},j/N^{k}]:\ A\cap [(j-1)/N^{k},j/N^{k}]\not=\ess \}.$$
Clearly, ${\can}_{N,k}(A)=\lll({\caf}_{N,k}(A))\cdot N^{k}.$
(We denote by $\lll$ the one-dimensional Lebesgue measure.)
\end{notation}


The open ball centered at $x$ and of radius $r$ is denoted by $B(x,r)$.
The $\aaa$-dimensional Hausdorff measure is denoted by $\cah^{\aaa}$.

We remind the reader to a few facts about iterated function systems. The details
can be found in many books, for example in \cite{Fa1}. If $S_{i}$, $i=1,...,N$, $N\geq 2$ is a family of similarities
with contraction ratios $c_{i}$ then we talk about an IFS, an iterated function
system. The attractor of the IFS is the compact set $F$ satisfying
$F=\cup_{i}S_{i}(F)$. The similarity dimension of $F$ is the unique $\aaa$
for which $\sum_{i}c_{i}^{\aaa}=1$. If the IFS satisfies the, so called open set condition
then the Hausdorff, the Minkowski and the similarity dimension of $F$ are the same
(see for example Theorem 9.3 of \cite{Fa1} ).
We recall 4.14. Theorem from \cite{[Ma]} by using our notation:
\begin{theorem}\label{*Ma414}
If $S_{i}$, $i=1,...,N$ satisfies the open set condition, then the invariant
set $F$ is self-similar and $0 < \cah^{\aaa}(F) < +\oo$, whence $\aaa = \dimh F$, where 
$\aaa$
is the unique number for which
$\ds \sum_{i}c_{i}^{\aaa}=1.$
Moreover, there are positive and finite numbers $a$ and $b$ such that
\begin{equation}\label{*est414*}
a r^{\aaa} < \cah^{\aaa}(F\cap B(x,r)) <b r^\aaa \text{ for }x \in F, \  0<r<1.
\end{equation}
\end{theorem}
In our paper the self similar sets for which this theorem will be applied will satisfy the
 strong separation condition,
that is the sets $S_{i}(F)$ will be disjoint, and hence the open set condition will
also be satisfied. 

We suppose that $F\sse [0,1]$ is the attractor of an IFS satisfying the open set condition.
Estimation \eqref{*est414*} will be useful several ways.

If one considers the function $\fff(x)=\cah^{\aaa}(F\cap [0,x])/\cah^{\aaa}(F)$ then it is easy to deduce from \eqref{*est414*} that $\fff$ is H\"older-$\aaa$.

On the other hand, using the lefthand-side inequality in \eqref{*est414*} it is also
easy to see that for fixed $N $ there exists a constant $C_{F}$ such that
\begin{equation}\label{*est414corr*}
\can_{N,k}(F)<C_{F} N^{\aaa k}\text{ for any }k\in\N.
\end{equation}
Indeed, one needs to observe that if $x\in F\cap [(j-1)/N^{k},j/N^{k}]$
then $\cah^{\aaa}(F\cap [(j-2)/N^{k},(j+1)/N^{k}])\geq
\cah^{\aaa}(F\cap B(x,1/N^{k})>a /N^{\aaa k}$.

Readers of earlier versions of this paper asked for more details about
``integrating" monotonicity results to obtain convexity estimates.
In the proof of the next proposition we provide these details.
\begin{proposition}\label{*integr*}
Suppose $\aaa\in [0,1]$, $g\in C[0,1]$ and $g|_{B}$ is not monotone for any
$B\sse [0,1]$ with $\udimm B>\aaa$. 
Put $f(x)=\int_{0}^{x}g(t)dt.$ Then $f|_{A}$ is not convex, nor concave for any
$A\sse [0,1]$ with $\udimm A>\aaa$.
\end{proposition}

\begin{proof}
Suppose that $A\sse [0,1]$ is closed and $f|_{A}$ is convex (the concave case is similar
and is left to the reader). Denote by $\cai_A$ the shortest closed interval containing $A$.
One can extend the definition of $f|_{A}$ onto $\cai_A$ to obtain a convex function $h$
defined on $\cai_{A}$ and $f|_{A}=h|_{A}$. At two-sided accumulation points
$a$ of $A$ we have $h'(a)=f'(a)=g(a)$. At one-sided accumulation points $a$ of $A$
we have $h'_{\pm}(a)=f'(a)=g(a)$ where $h'_{\pm}(a)=h'_{+}(a)$, or $h'_{\pm}(a)=h'_{-}(a)$ for right, or left accumulation points of $A$, respectively.

Suppose $a$ is an isolated point of $A$.
If $a$ is not an endpoint of $\cai_A$ then select $c\in A$ and $b\in A$ such that
$b<a<c$ and $(b,c)\cap A=\{ a \}.$ By the Mean Value Theorem
there is $a_{-}\in (b,a)$ and $a_{+}\in (a,c)$  such that $$f'(a_-)=g(a_{-})=\frac{f(a)-f(b)}{a-b}=\frac{h(a)-h(b)}{a-b}, \text{ and }$$
$$f'(a_+)=g(a_{+})=\frac{f(c)-f(a)}{c-a}=\frac{h(c)-h(a)}{c-a}.$$
If $a$ is the left-endpoint of $\cai_A$ then we define only $a_+$, if 
$a$ is the right-endpoint of $\cai_A$ then we define only $a_-$.

Denote by $B$ the set which contains all accumulation points of $A$
and the points $a_{+}$ and $a_{-}$ for isolated points of $A$.
Then the convexity of $h$ on $\cai_A$ and the above
equalities imply that $g$ is montone increasing on $B$ and hence,
say for $N=2$ we have
$$\udimm B=\limsup_{k\to\oo}\frac{\log \can_{2,k}(B)}{k\log 2}\leq \aaa.$$

The $1/2^{k}$ grid intervals taken into consideration in $\can_{2,k}(B)$
cover all accumulation points of $A$ and the points $a_{+}$ and $a_{-}$ 
corresponding to isolated points of $A$. 
Hence $\can_{2,k}(A)\leq  3\can_{2,k}(B)$.
This implies that $\udimm A\leq \aaa$.
\end{proof}

\section{Main result}

 \begin{theorem}\label{*5b*}
 Let $0<\aaa<1$. There exits $f\in C^{1+\aaa}[0,1]$
  such that for any $A\sse [0,1]$ with $\udimm A>\aaa$ the restriction
  $f|_{A}$ is neither convex, nor concave.
 \end{theorem}

\begin{remark}
Multiplying $f$ by a suitable constant one can achieve $f\in C_{1}^{1+\aaa}
[0,1]$ as well.
As it was mentioned in the introduction the theorem can be proved for $\aaa=0$
as well, with a suitably modified other, rather technical proof.
\end{remark}
 
 \begin{proof}
 We will define $g=f'\in C^{\aaa}[0,1]$.
 
 A large even integer $N$ will be fixed later.
 
 Before giving the details of the proof we give a list of some notation
 introduced at different steps. This might be helpful for later reference.
 
 In Subsection \ref{*subsfofffo} we introduce the self-similar set $F_{0}$
 and the H\"older $\aaa$ function $\fff_{0}$, which is constant on the
 connected components of the open set $G_{0}=[0,1]\sm F_{0}$. 
 The collection of these connected components of $G_{0}$ are denoted by $\cai_{0}$.
 
 In Subsection \ref{*subsfffn}
 the functions $\fff_{n}$ are defined. These functions
  are constant on the
 connected components ($I(a,b,j)$ in \eqref{*h6*d}) of the open sets $G_{n}=[0,1]\sm F_{n}$ defined in \eqref{*gndef*}.
 The connected components of $G_{n}$ are denoted by $\cai_{n}$.
 These sets are nested, $G_{n+1}\sse G_{n}$.
While $\fff_{n}$, defined in \eqref{*h6*b} is constant on the intervals $I(a,b,j)$ there will be some  transitional intervals, where $\fff_{n}$ is non-constant and linear.
These transitional intervals are the connected components of  $G_{n}'$,
defined in \eqref{*h7*a}.
The collection of these connected components of $G_{n}'$ are denoted by $\cai_{n}'$. The sets $G_{n}'$ are disjoint.
 
 In Subsection \ref{*ss33} the function $g_{1}$ is defined. Its transitional intervals, denoted by $\cai_{T,1}'$ are the connected components of
 $G'_{T,1}=\cup_{n=1}^{\oo}G_{n}'.$ 
 
 In Subsection \ref{*ss34} the functions $g_{n}$ are defined for $n\geq 2$.
 The corresponding transitional intervals are denoted by $\cai_{T,n}'$. They  are the connected components of
 $G'_{T,n}.$ 
 The sets $G'_{T,n}$ are also nested, they satisfy $G'_{T,n+1}\sse G'_{T,n}.$
 The function $g$ will be the limit of the functions $g_{n}$.
 
 In Subsection \ref{*ss35} the H\"older property of $g$ and of its antiderivative $f$ is verified.
 
 In Subsection \ref{*ss36} we define and estimate $\cas(G,\aaa)$ to measure the size of the sets $G_{n}'$. 
 
 In Subsection \ref{*ss37} we give the upper estimate of the upper box dimension of the sets $A$ on which $f$ can be convex, or concave.
Here the most important definition is $G''_{T,j,k}\sse G'_{T,j}$ which contains all transitional intervals of $g_{j}$ which are of length longer than $N^{-k}$.
The collection of the connected components of
$G''_{T,j,k}$ is denoted by $\cai''_{T,j,k}$.
In Claim \ref{*nnk*} we estimate ${\can}_{N,k}(A\sm {G}_{T,1,k}'')$
and in Claim \ref{*nnkj*} we reduce the general case to this initial one.

 \subsection{The definition of the self similar set $F_{0}$ and of the H\"older $\aaa$ function $\fff_{0}$}\label{*subsfofffo}
 \begin{figure}[h]
  \begin{center}
  \includegraphics[width=14cm]{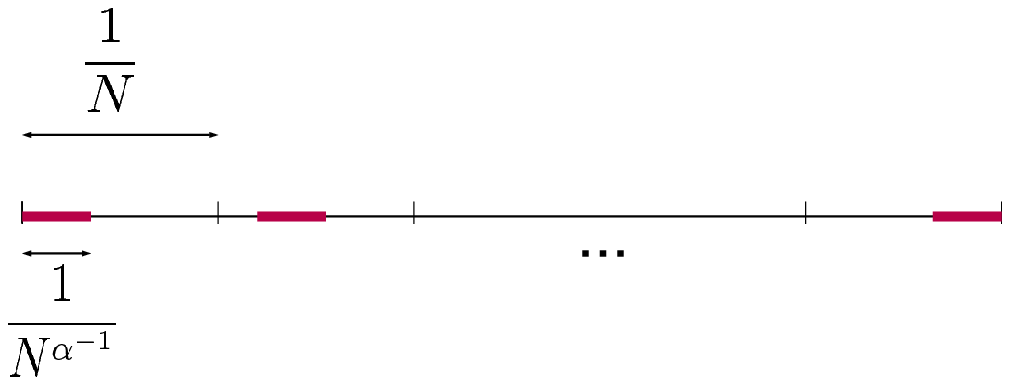}
    \caption{$F_{0,1}$}
     \label{fig2}
  \end{center}
\end{figure}
 First we select a standard H\"older $\aaa$ function $\fff_{0}$ the following way.
 Let $F_{0,1}$ consist of the following subintervals of
$[0,1]$ 
 \begin{equation}\label{*subi}
 \Big [ (j-1)\cdot \left (\frac{1}{N}+\frac{1}{N-1}\left (\frac{1}{N}-\frac{1}{N^{\aaa^{-1}}} \right ) \right ), \hskip2cm
 \end{equation}
 $$
 \hskip2cm
 (j-1)\cdot \left (\frac{1}{N}+\frac{1}{N-1}\left (\frac{1}{N}-\frac{1}{N^{\aaa^{-1}}} \right ) \right )
+\frac{1}{N^{\aaa^{-1}}} \Big ],\ j=1,...,N .$$
We remark that this complicated looking definition implies that
 $0$ and $1$ both belong to a component interval of  
 $F_{0,1}$. These component intervals are of equal
 length $1/N^{\aaa^{-1}}$
 and are equally spaced in $[0,1].$
We denote by $F_{0}$ the self-similar set one can obtain by repeating the steps
used for $F_{0,1}$ in each subinterval infinitely often. That is, we take the attractor 
of the IFS mapping linearly $[0,1]$ onto the components of $F_{0,1}$.

We can apply Theorem \ref{*Ma414} and the subsequent remarks
to $F_{0}$.

The IFS defining $F_{0}$ consists of $N$ similarities each of ratio ${1}/{N^{\aaa^{-1}}}$ and hence $N\cdot ({1}/{N^{\aaa^{-1}}})^{\aaa}=1$.
The similarity dimension and the other dimensions 
of $F_{0}$ coincide and hence we have $\dimm F_{0}=\udimm F_{0}=\aaa$ 
and $0<\cah^{\aaa}(F_{0})<\oo$ and
$$\fff_{0}(x)=\cah^{\aaa}(F_{0} \cap [0,x])/\cah^{\aaa}(F_{0})$$
satisfies with a suitable constant $C_{\fff_{0}}$
\begin{equation}\label{*h4*a}
|\fff_{0}(x)-\fff_{0}(y)|\leq C_{\fff_{0}}|x-y|^{\aaa} \text{ for all }x,y\in [0,1],
\end{equation} 
that is, $\fff_{0}$ is a H\"older-$\aaa$ function and it is constant on the
intervals contiguous to $F_{0}$. We also have $0,1\in F_{0}$.
Later in \eqref{*h19*a} we will also make the additional assumption that $C_{\fff_{0}}>2$.

We put ${G}_{0}=[0,1] \sm F_{0}$ and denote by ${\cal I}_{0}$ the system of its component
intervals, that is,
$${\cal I}_{0}=\{ (a,b):(a,b)\text{ is a connected component of }{G}_{0} \}.$$
This is the system of intervals contiguous to $F_{0}$.

Since $\dimh F_{0}=\udimm F_{0}=\aaa$  by 
the remarks after Theorem \ref{*Ma414}
there exists a constant $C_{F_{0}}$  such that 
\begin{equation}\label{*h5*b}
{\can}_{N,k}(F_{0})\leq C_{F_{0}} N^{\aaa k}\text{ for any }k\in\N,
\end{equation}
 and, obviously
 \begin{equation}\label{*h5*c}
 {\caf}_{N,k}(F_{0})\text{ contains all }(a,b)\in {\cal I}_{0}\text{ for which }b-a\leq\frac{1}{N^{k}}.
 \end{equation}
 It is also clear that  for sufficiently large $N$
 \begin{equation}\label{*h5*a}
b-a<\frac{1}{N}+\frac{1}{N(N-1)}=\frac{1}{N-1}
<\frac{2^{\aaa/2}}{N}<\frac{2}{N}
\text{ for any $(a,b)\in {\cal I}_{0}$.}
 \end{equation}
  
 The system of the, so called, transitional intervals ${\cal I}_{0}'=\ess$ and
 $G_{0}'=\ess$,  by definition at this initial step of our construction.

 \subsection{  Definition of the functions $\fff_{n}$ for $n\geq 1$}\label{*subsfffn} 
 We want to define a sequence of
 continuous
functions $\fff_{n}$ by induction. Suppose that we have already defined the function
$\fff_{n-1}$, the open set ${G}_{n-1}$ and $F_{n-1}=[0,1]\sm {G}_{n-1}$.
The system of component intervals of ${G}_{n-1}$ is denoted by ${\cal I}_{n-1}$
and for any $(a,b)\in {\cal I}_{n-1}$ we have 
\begin{equation}\label{*h6*a}
b-a<\frac{2^{\aaa/2}}{N^{n}}<\frac{2}{N^{n}} \text{ and } \fff_{n-1}\text{ is constant on }[a,b].
\end{equation}
\begin{figure}[h]
  \begin{center}
  \includegraphics[width=12cm]{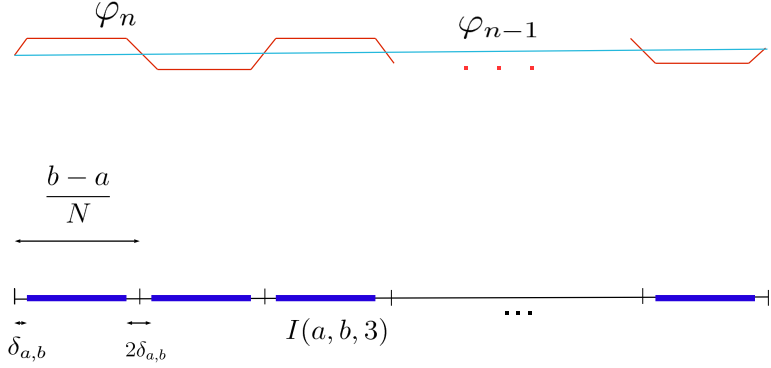}
    \caption{The intervals $I(a,b,j)$ and the function $\fff_n$}
     \label{fig3}
  \end{center}
\end{figure}
 We also have the system of ``transitional intervals", ${\cal I}_{n-1}'$.
 \begin{equation}\label{*h6*c}
 \text{ We put $\fff_{n}(x)=\fff_{n-1}(x)$ for $x\in F_{n-1}$.}
 \end{equation}
 Suppose $(a,b)\in {\cal I}_{n-1}$.
 Let
 \begin{equation}\label{*h6*d}
 I(a,b,j)=\Big  [ a+\frac{j-1}{N}(b-a)+\frac{(b-a)^{2\aaa^{-1}}}{N^{2\aaa^{-1}}},
  a+\frac{j}{N}(b-a)-\frac{(b-a)^{2\aaa^{-1}}}{N^{2\aaa^{-1}}}\Big ]
 \end{equation}
 $j=1,...,N$.
 If we let $\ddd_{a,b}=\frac{(b-a)^{2\aaa^{-1}}}{N^{2\aaa^{-1}}}$
 then
 the intervals $I(a,b,j)$ are equally placed within $(a,b)$ with gaps 
$2\ddd_{a,b}$ separating two consecutive such intervals, and
there is a gap of length $\ddd_{a,b}$ before the first and after the last such interval,
see the bottom half of Figure \ref{fig3}.
These gaps will be called later transitional intervals, due to the fact
that the functions $\fff_{n}$ will take constant values on the intervals
$I(a,b,j)$ and will change its values linearly on the transitional
intervals. 
 If $(a,b)\in{\cal I}_{n-1}$ and $x\in I(a,b,j)$ for a $j\in \{ 1,...,N \}$
 then put
 \begin{equation}\label{*h6*b}
 \fff_{n}(x)=\fff_{n-1}(a)+(-1)^{j+1}\frac{(b-a)^{2\aaa^{-1}}}{N^{2\aaa^{-1}}}.
 \end{equation}
 Set 
 \begin{equation}\label{*h7*a}
 {G}_{n}'=\bigcup_{(a,b)\in{\cal I}_{n-1}} ((a,b)\sm\cup_{j=1}^{N}I(a,b,j)).
 \end{equation}
 Observe that we have not defined yet the function $\fff_{n}$ on ${G}_{n}'$.
 The system of component intervals of ${G}_{n}'$ will be denoted by ${\cal I}_{n}'$
 and we call them transitional intervals.
 It will be useful later that our construction implies that
 if $(a,b)\in{\cal I}_{n}'$ is a transitional interval then
 \begin{equation}\label{*h7*aa}
 \Big | \frac{\fff_{n}(b)-\fff_{n}(a)}{b-a}\Big |=1.
 \end{equation}
 We define $\fff_{n}$ on $(a,b)\in{\cal I}_{n}'$ so that it is linear and connects 
 $(a,\fff_{n}(a))$ and $(b,\fff_{n}(b))$.
 
 Set 
 \begin{equation}\label{*gndef*}
 {G}_{n}=\bigcup_{(a,b)\in{\cal I}_{n-1}}\bigcup_{j=1}^{N}\inte(I(a,b,j))
 \end{equation}
 and
 ${\cal I}_{n}$ will denote the set of component intervals of ${G}_{n}$.
 We put $F_{n}=[0,1]\sm G_{n}$.
 Clearly, by \eqref{*h6*a} and \eqref{*h6*d}
 \begin{equation}\label{*h7*c}
 b-a<\frac{2^{\aaa/2}}{N^{n+1}}\text{ for all }(a,b)\in{\cal I}_{n}.
 \end{equation}

 \begin{claim}\label{*holdfi*}
 We have
 \begin{equation}\label{*h18*a}
 |\fff_{n}(x)-\fff_{n}(y)|\leq C_{\fff_{0}}(1+...+\frac{1}{2^n})|x-y|^{\aaa},
 \text{ for all }x,y\in[0,1].
 \end{equation}
 \end{claim}
 
 \begin{proof}[Proof of Claim \ref{*holdfi*}]
 By \eqref{*h4*a} this holds for $\fff_{0}$. By induction, we can suppose that
 \eqref{*h18*a} holds for $n-1$ instead of $n$.
 Without limiting generality suppose that $x<y$.
 If $x,y\in F_{n-1}$ then $\fff_{n-1}(x)=\fff_{n}(x)$
 and $\fff_{n-1}(y)=\fff_{n}(y)$ and our induction hypothesis implies \eqref{*h18*a}.
 Suppose $x\in(a_{x},b_{x})\in {\cal I}_{n-1}$ 
 and $y\in(a_{y},b_{y})\in {\cal I}_{n-1}$.
 (The other two cases, $x\in F_{n-1}$ ,  $y\in(a_{y},b_{y})\in {\cal I}_{n-1}$
 and  $x\in(a_{x},b_{x})\in {\cal I}_{n-1}$ 
 and $y\in F_{n-1}$ can be treated analogously and we omit the details.)
 
 The definition of $\fff_{n}$ implies that
 \begin{equation}\label{*h18*b}
 |\fff_{n}(b_{x})-\fff_{n}(x)|\leq |b_{x}-x|\text{ and }
 |\fff_{n}(y)-\fff_{n}(a_{y})|\leq |y-a_{y}|.
 \end{equation}
 We can suppose that $N$ is so large that
 \begin{equation}\label{*h19*a}
 N^{1-\aaa}>4 \text{ and }C_{\fff_{0}}>2.
 \end{equation}
 By \eqref{*h6*a}, $|b_{x}-a_{x}|<\frac{2}{N^{n}}$
 and $|b_{y}-a_{y}|<\frac{2}{N^{n}}$.
 This and \eqref{*h19*a} imply
 \begin{equation}\label{*h19*b}
 |b_{x}-x|^{1-\aaa}<\frac{2^{1-\aaa}}{N^{(1-\aaa)n}}<\frac{C_{\fff_{0}}}{2^{n+1}}\text{ and }
 |y-a_y|^{1-\aaa}<\frac{2^{1-\aaa}}{N^{(1-\aaa)n}}<\frac{C_{\fff_{0}}}{2^{n+1}}.
 \end{equation}
 Using that $0<\aaa<1$ and $b_{x},a_{y}\in F_{n-1}$ we infer
 $$|\fff_{n}(x)-\fff_{n}(y)|\leq
 |\fff_{n}(b_{x})-\fff_{n}(x)|+|\fff_{n-1}(b_{x})-\fff_{n-1}(a_{y})|
 +|\fff_{n}(y)-\fff_{n}(a_{y})|\leq$$
 $$\frac{C_{\fff_{0}}}{2^{n+1}}|b_{x}-x|^{\aaa}+
 C_{\fff_{0}}(1+...+\frac{1}{2^{n-1}})
 |b_{x}-a_{y}|^{\aaa}+ \frac{C_{\fff_{0}}}{2^{n+1}}
 |y-a_{y}|^{\aaa}<$$ $$
 C_{\fff_{0}}(1+...+\frac{1}{2^{n}})|x-y|^{\aaa}.$$
  \end{proof}
 
 If $x\in(a,b)\in{\cal I}_{0}$
 then
 by \eqref{*h6*b}
 \begin{equation}\label{*h24*aa}
 |\fff_1(x)-\fff_0(x)|\leq \frac{(b-a)^{2\aaa^{-1}}}{N^{2\aaa^{-1}}}
 <\frac{(b-a)^{2}}{N^{2}}
 \end{equation}
 and in general, using that for $(a',b')\in {\cal I}_{n-1},$
 $(a',b')\sse (a,b)\in {\cal I}_{0}$ we have $b'-a'<(b-a)/N^{n-1}$ we infer
 \begin{equation}\label{*h24*bb}
 |\fff_n(x)-\fff_{n-1}(x)|\leq \frac{(b-a)^{2\aaa^{-1}}}{N^{(n-1)2\aaa^{-1}}}\cdot \frac{1}{N^{2\aaa^{-1}}}<\frac{(b-a)^{2}}{N^{2(n-1)}}\cdot \frac{1}{N^{2}}
 \end{equation}
 for $n=2,3,...$.

 \subsection{  Definition of the function $g_{1}$}\label{*ss33} 

 From \eqref{*h24*bb} it follows that $\fff_{n}(x)$ converges uniformly
to a continuous function $g_{1}(x)$.  

We infer from \eqref{*h18*a} that
\begin{equation}\label{*h8*d}
|g_{1}(x)-g_{1}(y)|\leq 2C_{\fff_{0}}|x-y|^{\aaa}\text{ for all }x,y\in[0,1],
\end{equation}
 moreover  by \eqref{*h24*bb} for $x\in(a,b)\in {\cal I}_{0}$
 we have 
 \begin{equation}\label{*h25*aa}
 |g_{1}(x)-\fff_{1}(x)|\leq \frac{(b-a)^{2}}{N^{2}}\Big (\frac{1}{N^{2}}+\frac{1}{N^{4}}+... \Big )<\frac{(b-a)^{2}}{N^{2}}\cdot \frac{2}{N^{2}}.
 \end{equation}
 From \eqref{*h6*b}, \eqref{*h24*aa} and \eqref{*h24*bb} we also obtain
 for $x\in (a,b)\in {\cal I}_{0}$
 \begin{equation}\label{*h25*gfo*}
 |g_{1}(x)-\fff_{0}(x)|<\frac{(b-a)^{2\aaa^{-1}}}{N^{2\aaa^{-1}}}\Big (1+\frac{1}{N^{2}}+... \Big )<\frac{(b-a)^{2\aaa^{-1}}}{N^{2\aaa^{-1}}}2^{\aaa^{-1}}
 \end{equation}
 if $N$ is sufficiently large.

\begin{claim}\label{*clgo}
We have
\begin{equation}\label{*h8*b}
0< g_{1}(x)< 1 \text{ for any }x\in(0,1), \  g_{1}(0)=0\text{ and }
g_{1}(1)=1.
\end{equation}
\end{claim}
 
 \begin{proof}[Proof of Claim \ref{*clgo}]
 Since $0,1\in F_{0}$ we have
 $g_{1}(0)=\fff_{0}(0)=0$ and
 $g_{1}(1)=\fff_{0}(1)=1$.
 
 The self-similarity of $F_{0}$ and  \eqref{*h5*a} imply that if $(a,b)\in{\cal I}_{0}$ satisfies

$$\text{ $(a,b)\sse
 [N^{-\aaa^{-1}(k+1)},N^{-\aaa^{-1} k}]$, $k\in \N$ then}$$ 
 \begin{equation}\label{*Nak*}
 b-a<\frac{2^{\aaa/2}}{N\cdot N^{\aaa^{-1} k}}.
 \end{equation}
Using \eqref{*h25*gfo*} for $x\in(a,b)\in {\cal I}_{0}$
 with sufficiently large $N$ we obtain
 \begin{equation}\label{*h26*gefo*}
 |g_{1}(x)-\fff_{0}(x)|<\frac{2N^{-2\aaa^{-1}} N^{-2\aaa^{-2}k}}{N^{2\aaa^{-1}}}\cdot 2^{\aaa^{-1}}<
 \frac{1}{N^{k+1}}\cdot \frac{1}{10}.
 \end{equation}
 Using self-similarity of $F_{0}$, of $\fff_{0}$ and translation invariance of the Hausdorf measure we obtain $$\fff_{0}(x)\geq 
 \fff_{0}\Big 
 (\frac{1}{N^{(k+1)\aaa^{-1}}} \Big )=
 \Big 
 (\frac{1}{N^{(k+1)\aaa^{-1}}} \Big )^{\aaa}=\frac{1}{N^{k+1}}$$
 and hence $g_{1}(x)>0$.
 
 For $x\in F_{0}$, $g_{1}(x)=\fff_{0}(x)$ and we obtain that for
 $0<x\leq 1$ we have $g_{1}(x)>0$. 
 
 Similarly, one can see that 
 $g_{1}(x)<1$ holds for $0\leq x <1$.
 This implies \eqref{*h8*b}.
 \end{proof}


We put
\begin{equation}\label{*h8*c}
{G}_{T,1}'=\bigcup_{n=1}^{\oo}{G}_{n}'
\end{equation}
and denote by ${\cal I}_{T,1}'$ the system of component intervals of ${G}_{T,1}'$.
It is easy to see that ${\cal I}_{T,1}'=\cup_{n=1}^{\oo}{\cal I}_{n}'$.

If $(a,b)\in {\cal I}_{T,1}'$, that is, $(a,b)$ is a transitional interval for
$g_{1}$ then by \eqref{*h7*aa} it is linear on $(a,b)$ with slope
of absolute value $1$.

 \subsection{  Definition of the functions $g_{n}$, $n\geq 2$ and of $g$}\label{*ss34}
For any $(a,b)$ we denote by
$\FFF_{(a,b)}$ the linear mapping $\FFF_{(a,b)}:[a,b]\to [0,1]$,
$\FFF_{(a,b)}(x)=\frac{x-a}{b-a}.$
 
 To define $g_{2}$ for $x\in{F}_{T,1}'=[0,1]\sm{G}_{T,1}'$ we put $g_{2}(x)=g_{1}(x)$.
 To obtain $g_{2} $ on the transitional intervals we modify $g_{1}$. For $x\in {G}_{T,1}'$ if $x\in(a,b)\in{\cal I}_{T,1}'$ then we put
 \begin{equation}\label{*h9*a}
 g_{2}(x)=g_{1}(a)+(g_{1}(b)-g_{1}(a))\cdot g_{1}\Big (\frac{x-a}{b-a}\Big )=
 \end{equation}
 $$g_{1}(a)+(g_{1}(b)-g_{1}(a))\cdot g_{1}\circ\FFF_{(a,b)}(x).$$
 Recall that  $|g_{1}(b)-g_{1}(a)|=|b-a|$.
 We put ${G}_{T,2}'=\cup_{(a,b)\in{\cal I}_{T,1}'}\FFF_{(a,b)}^{-1}({G}_{T,1}')$
 and ${F}_{T,2}'=[0,1]\sm{G}_{T,2}'$.
 The system of component intervals of ${G}_{T,2}'$ is denoted by 
 ${\cal I}_{T,2}'$.
One can easily see that $(a,b)\in{\cal I}_{T,2}'$ if  there exists $(a',b')\in{\cal I}_{T,1}'$
and $(a'',b'')\in{\cal I}_{T,1}'$  such that $(a,b)=\FFF_{(a',b')}^{-1}(a'',b'')$.
One can also see that for $(a,b)\in \cai_{T,2}'$ we have $|g_{2}(b)-g_{2}(a)|=|b-a|$ and $g_{2}$ is linear on these intervals.

From \eqref{*h5*a}, \eqref{*h6*d}  and  \eqref{*h7*a} it follows that
\begin{equation}\label{*h10*a}
b'-a'<\left (\frac{2^{\aaa/2}}{N} \right )^{2\aaa^{-1}}\cdot \frac{2}{N^{2\aaa^{-1}}}<\frac{4}{N^{4}}\text{ for any }(a',b')\in{\cal I}_{T,1}'.
\end{equation}

To define the functions $g_{n}$ we proceed again by induction.
Suppose that $g_{n-1}$ has been already defined and the open set 
${G}_{T,n-1}'$ consists of the transitional intervals of $g_{n-1}$.
The system of these component intervals of ${G}_{T,n-1}'$ is denoted by
${\cal I}_{T,n-1}'$.
We suppose that
\begin{equation}\label{*h10*aa}
b'-a'<\frac{4^{n-1}}{N^{4(n-1)}}\text{ holds for any }(a',b')\in{\cal I}_{T,n-1}'.
\end{equation} 
We also assume that $g_{n-1}$ is linear on the transitional intervals
$(a',b')\in{\cal I}_{T,n-1}'$ and its slope is of absolute value $1$ on these intervals.

To define $g_{n}$ for $x\in {F}_{T,n-1}'=[0,1]\sm{G}_{T,n-1}'$ we put
$g_{n}(x)=g_{n-1}(x).$ For $x\in{G}_{T,n-1}'$  there exists $(a,b)\in{\cal I}_{T,n-1}'$
 such that $x\in(a,b)$.
 Let
 \begin{equation}\label{*h12*c}
 g_{n}(x)=g_{n-1}(a)+(g_{n-1}(b)-g_{n-1}(a))\cdot g_{1}\Big (\frac{x-a}{b-a}\Big )=
 \end{equation}
 $$g_{n-1}(a)+(g_{n-1}(b)-g_{n-1}(a))\cdot g_{1}\circ\FFF_{(a,b)}(x).$$
 We put 
 \begin{equation}\label{*gvtn*}
 {G}_{T,n}'=\bigcup_{(a,b)\in{\cal I}_{T,n-1}'}\FFF_{(a,b)}^{-1}({G}_{T,1}')
 \text{ and }{F}_{T,n}'=[0,1]\sm{G}_{T,n}'.
 \end{equation}
 The system of component intervals of ${G}_{T,n}'$ is denoted by 
 ${\cal I}_{T,n}'$.

From \eqref{*h10*a}, \eqref{*h10*aa}, \eqref{*h12*c} and  \eqref{*gvtn*} it follows that
\begin{equation}\label{*h12*aaa}
b'-a'<\frac{4^{n}}{N^{4n}}\text{ for any }(a',b')\in{\cal I}_{T,n}'.
\end{equation} 
Observe that by \eqref{*h12*c} for any $(a',b')\in{\cal I}_{T,n}'$ the function $g_{n}$ is linear with slope of
 absolute value $1$.
 
 From \eqref{*h8*b},  \eqref{*h10*aa} and  \eqref{*h12*c} it follows that
 $g_{n}(x)$ converges uniformly to a continuous function $g(x)$.
 
 From our construction, especially from \eqref{*h12*c} it follows that
$g$ satisfies  a (restricted) self-similarity property.
For any $n\in\N$ and $(a,b)\in {\cal I}_{T,n}'$
we repeat on $(a,b)$ the construction steps of $g$ 
scaled down by the factor $(b-a)$. Hence
\begin{equation}\label{*h12*cc}
g(x)=g(a)+(g(b)-g(a))\cdot g\circ \FFF_{(a,b)}(x)\text{ for }x\in (a,b)
\end{equation} 
where $(a,b)$ can be any transitional interval, that is $(a,b)\in\cup_{n=1}^{\oo}
{\cal I}_{T,n}'$.

\begin{claim}\label{*claimh28*}
If $N>100$ and $x\in (a,b) \sm {G}_{1}'$, $(a,b)\in{\cal I}_{0}$ then
\begin{equation}\label{*h21*b}
|g(x)-g_{1}(x)|<4\frac{(b-a)^{2\aaa^{-1}}}{N^{4\aaa^{-1}}}<4 \frac{(b-a)^{2}}{N^{4}}.
\end{equation}
\end{claim}

\begin{proof}[Proof of Claim \ref{*claimh28*}]
If $x\in (a,b)\sm {G}_{1}'$ then  there exists $j\in \{ 1,...,N \}$
 such that $x\in I(a,b,j)$. The length of the $I(a,b,j)$ intervals is less than
 $(b-a)/N$ and hence if $(a',b')\in {\cal I}'_{n}$ is a component of ${G}_{T,1}'$
 in $I(a,b,j)$ then
 \begin{equation}\label{*h20*a}
 b'-a'\leq 2\frac{(b-a)^{2\aaa^{-1}}}{N^{2\aaa^{-1}}}\cdot \frac{1}{N^{2\aaa^{-1}}}.
 \end{equation}
 By \eqref{*h8*b}, \eqref{*h9*a} and \eqref{*h20*a}
 \begin{equation}\label{*h28*aa}
 |g_{2}(x)-g_{1}(x)|\leq |b'-a'|<2\frac{(b-a)^{2\aaa^{-1}}}{N^{4\aaa^{-1}}}< \frac{2(b-a)^{2\aaa^{-1}}}{N^{4}}.
 \end{equation}
 In general, suppose that $j\geq 2$ is given and if
 $(a',b')\in G'_{T,j-1}$ is the component containing $x$
 then 
 \begin{equation}\label{*40*bb}
 b'-a'<2\frac{(b-a)^{2\aaa^{-1}}}{N^{4\aaa^{-1}}}\cdot \left (\frac 4{N^{4}} \right)^{j-2}.
 \end{equation}
 If $x\in {F}_{T,j}'$ then $g_{j+1}(x)=g_{j}(x)$. 
 Consider the case $x\in {G}_{T,j}'$.
 By \eqref{*h7*a}, \eqref{*h10*a}, \eqref{*gvtn*} and \eqref{*h20*a} if $(a'',b'')$ is a component of
 ${G}_{T,j}'$ containing $x$ then 
 \begin{equation}\label{*h20*b}
 b''-a''<|b'-a'|\cdot \frac{4}{N^{4}}\leq  2\frac{(b-a)^{2\aaa^{-1}}}{N^{4\aaa^{-1}}}
 \cdot \left (\frac 4{N^{4}} \right)^{j-1}.
 \end{equation}
 As we obtained \eqref{*h28*aa} this implies 
 \begin{equation}\label{*h21*a}
 |g_{j+1}(x)-g_{j}(x)|\leq |b''-a''|< 2\frac{(b-a)^{2\aaa^{-1}}}{N^{4\aaa^{-1}}}
 \cdot \left (\frac 4{N^{4}} \right)^{j-1}
 .
 \end{equation}
 Repeating the above argument we infer that for $x\in(a,b)\sm{G}_{1}'$ we have
 \begin{equation}\label{*h21*bb}
 |g(x)-g_{1}(x)|\leq \sum_{j=1}^{\oo} |g_{j+1}(x)-g_{j}(x)|<
 \end{equation}
 $$2\frac{(b-a)^{2\aaa^{-1}}}{N^{4\aaa^{-1}}} \Big ( 1+\frac{4}{N^{4}}+\frac{4^{2}}{N^{8}}+... \Big )< 4\frac{(b-a)^{2\aaa^{-1}}}{N^{4\aaa^{-1}}}.$$
This completes the proof of Claim \ref{*claimh28*}.
\end{proof}

\subsection{  Definition and H\"older property of $f$}\label{*ss35} We put $f(x)=\int_{0}^{x}
g(t)dt.$

In the rest of the proof we need to verify that $f$ has the properties claimed in Theorem \ref{*5b*}.
\begin{claim}\label{*ghaaa*}
The function $g$ is in $C^{\aaa}[0,1]$ and hence $f\in C^{1+\aaa}[0,1]$.
\end{claim}

 \begin{proof}[Proof of Claim \ref{*ghaaa*}.]
 First we prove a very special case of this claim. Namely, we show that  there exists 
 a constant $C_{g}'$  such that 
 \begin{equation}\label{*h14*aa}
 |g(x)-g(0)|< C_{g}' |x|^{\aaa}\text{ and }
 |g(1)-g(x)|< C_{g}' |1-x|^{\aaa}.
 \end{equation}
 Suppose $x\in[0,1]$. If $x\in{F}_{T,1}'$ then $g(x)=g_{1}(x)$ and from
 $0\in {F}_{T,1}'$ and \eqref{*h8*d} it follows that
 \begin{equation}\label{*h15*aa}
 |g_{1}(x)-g_{1}(0)|=|g(x)-g(0)|\leq 2 C_{\fff_{0}}|x|^{\aaa}.
 \end{equation}
 
 Suppose $x\in {G}_{T,1}'=[0,1]\sm {F}_{T,1}'.$
 The definition of $F_{0}$ implies that the largest component
 of ${G}_{0}$ in $[N^{-\aaa^{-1}},1-N^{-\aaa^{-1}}]$
is of length less than $ 1/(N-1)< \frac{2^{\aaa/2}}{N}$.
 By \eqref{*h6*d}  and by induction the largest component of any ${G}_{n}$
for any $n \geq 1$ is of length less than $1/((N-1)N^{n})<\frac{2^{\aaa/2}}{N^{n+1}}<2^{\aaa/2}/N$.
This implies by \eqref{*h6*d} that if $(a',b')\sse [N^{-\aaa^{-1}},1-N^{-\aaa^{-1}}]$ is a transitional interval, that is $(a',b')\in {\cal I}_{T,1}'$
then $b'-a'\leq 4\cdot N^{-4\aaa^{-1}}$.
Hence, for sufficiently large $N$
\begin{equation}\label{*h16*aa}
\frac{1}{N^{\aaa^{-1}}}\leq a'\text{ and }b'-a'\leq \frac{4}{N^{4\aaa^{-1}}}\leq \frac{4}{N^{3\aaa^{-1}}}\cdot a'<a'.
\end{equation}

By self-similarity of $F_{0}$ one can see analogously
by using \eqref{*Nak*}, that if $(a',b')\sse [N^{-(k+1)\aaa^{-1}},N^{-k\aaa^{-1}}]$, $k=1,...$ and $(a',b')\sse {\cal I}_{T,1}'$ then
for sufficiently large $N$
$$b'-a'< \frac{2}{N^{2\aaa^{-1}}}\cdot  \frac{2}{N^{2\aaa^{-1}} 
N^{2\aaa^{-2}k}}< \frac{1}{N^{(2k+2)\aaa^{-1}}}$$
and
\begin{equation}\label{*h17*aa}
\frac{1}{N^{(k+1)\aaa^{-1}}}\leq a'\text{ and }b'-a'\leq
\frac{1}{N^{(2k+2)\aaa^{-1}}}< \frac{1}{N^{k \aaa^{-1}}}a'<a'.
\end{equation}
 Similar estimates are also valid at the other end
 of $[0,1]$, that is if $(a',b')\sse [1-N^{-(k+1)\aaa^{-1}},1-N^{-k\aaa^{-1}}]$ and $(a',b')\in{\cal I}_{T,1}'$ then
 \begin{equation}\label{*h17*bb}
 \frac{1}{N^{(k+1)\aaa^{-1}}}\leq 1-b'\text{ and }
 b'-a'\leq 1-b'.
 \end{equation}
 We know from \eqref{*h8*d} that $g_{1}$ is H\"older $\aaa$
 with constant $2C_{\fff_{0}}$ on $[0,1]$.
 
 Suppose $x\in[0,1]$. If $x\in{F}_{T,1}'$ then $g_{1}(x)=g(x)$
 and hence
 \begin{equation}\label{*h17*cc}
 |g(x)|=|g(0)-g(x)|\leq
 2C_{\fff_{0}}|x|^{\aaa}\text{ and }
|g(x)-1|=|g(x)-g(1)|\leq 2C_{\fff_{0}}|1-x|^{\aaa}.
 \end{equation}
 If $x\in{G}_{T,1}'$ then  there exists $(a',b')\in {\cal I}_{T,1}'$
  such that $x\in(a',b')$. We also know that
  $|g(b')-g(a')|=|b'-a'|$ and for $x\in(a',b')$ we have
  $\min\{ g_{1}(a'),g_{1}(b') \}\leq g(x)\leq \max\{ g_{1}(a'),g_{1}(b') \}.$
That is, by \eqref{*h16*aa} and \eqref{*h17*aa}
\begin{equation}\label{*h18*ca}
|g(x)-g_{1}(a')|\leq |b'-a'|<a'.
\end{equation}
Since $|g(a')|=|g_{1}(a')|\leq 2C_{\fff_{0}}|a'|^{\aaa}\leq 2C_{\fff_{0}}|x|^{\aaa}$ and
$|g(x)-g_{1}(a')|<a'<|x|$
we have $|g(x)-g(0)|=|g(x)|\leq(2C_{\fff_{0}}+1)|x|^{\aaa}$.

A similar argument can show that
$|g(1)-g(x)|=|1-g(x)|\leq (2C_{\fff_{0}}+1)|1-x|^{\aaa}$.
Hence we can select $C_{g}'=2C_{\fff_{0}}+1$ in \eqref{*h14*aa}.

Next we show that $g$ is H\"older $\aaa$. Suppose that $x,y\in[0,1]$, $x<y.$
If $x,y\in{F}_{T,1}'$ then $g(x)=g_{1}(x)$, $g(y)=g_{1}(y)$ and by \eqref{*h8*d}
we have 
\begin{equation}\label{*h19*aa}
|g(x)-g(y)|\leq 2C_{\fff_{0}}|x-y|^{\aaa}.
\end{equation} 
One needs to consider several more cases.
We discuss in detail the case when $x,y\not\in{F}_{T,1}'$, that is $x,y\in{G}_{T,1}'$.
The other cases when one of $x$ and $y$ belongs to ${F}_{T,1}'$ and the other to
${G}_{T,1}'$ are analogous and are left to the reader.
The sets ${G}_{T,j}'$ are nested and we can choose $j$  such that $x$ and $y$
are not in the same component of ${G}_{T,j}'$ but $x$ and $y$ belong to the
same component of ${G}_{T,j'}'$ for $j'<j$.

If $j=1$ then set $(a,b)=(0,1)$ if $j>1$ then denote by $(a,b)$ the component
of ${G}_{T,j-1}'$ containing $x$ and $y$. By  \eqref{*h12*c} and  \eqref{*h12*cc}, $g_{j}|_{(a,b)}$
is a scaled down, similar copy of $g_{1}|_{(0,1)}$ and $g|_{(a,b)}$
is similar to $g|_{(0,1)}$. 
Hence, we can suppose that $j=1$ and $(a,b)=(0,1)$. 
Thus $x\in (a_{x},b_{x})\in {\cal I}_{T,1}'$, $y\in (a_{y},b_{y})\in {\cal I}_{T,1}'$
and $b_{x}<a_{y}$.
Since $b_{x},a_{y}\in {F}_{T,1}'$ we have
\begin{equation}\label{*h20*aa}
|g(b_{x})-g(a_{y})|=|g_{1}(b_{x})-g_{1}(a_{y})|\leq
2C_{\fff_{0}}|b_{x}-a_{y}|^{\aaa}<2C_{\fff_{0}}|x-y|^{\aaa}.
\end{equation}
Again by \eqref{*h12*cc}, $g|_{(a_{x},b_{x})}$ is similar to $g|_{(0,1)}$
hence, by what we have already shown in \eqref{*h14*aa}
$$|g(x)-g(b_{x})|=(b_{x}-a_{x})|g(\FFF_{(a_{x},b_{x})}(x))-g(1)| \
\leq (b_{x}-a_{x}) C_{g}'|\FFF_{(a_{x},b_{x})}(x)-1|^{\aaa}
\leq
$$ $$
 (b_{x}-a_{x}) C_{g}'\Big |\frac{x-a_{x}}{b_{x}-a_{x}}-\frac{b_{x}-a_{x}}{b_{x}-a_{x}}\Big |^{\aaa}=
 |b_{x}-a_{x}|^{1-\aaa} C_{g}'|x-b_{x}|^{\aaa}<$$ $$
 C_{g}'|x-b_{x}|^{\aaa}<
(2C_{\fff_{0}}+1)|y-x|^{\aaa}$$
and similarly
$$|g(y)-g(a_{y})|\leq (2C_{\fff_{0}}+1)|y-x|^{\aaa}.$$
Thus $|g(x)-g(y)|\leq (6C_{\fff_{0}}+2)|y-x|^{\aaa}.$
 \end{proof}

 \subsection{  Estimates of the size of the transitional intervals}\label{*ss36}
 \begin{definition}\label{*defSGa*}
 For an open set $G\sse [0,1]$ if $\cai_{G}$ denotes the system of component
 intervals of $G$ we put
 $$\cas(G,\aaa)=\sum_{(a,b)\in \cai_{G}}|b-a|^{\aaa}.$$
 If $\aaa<1$ then it is clear that $\lll(G)\leq \cas(G,\aaa).$
 \end{definition}
 
 By \eqref{*h6*a}, \eqref{*h6*d} and \eqref{*h7*a} for any $(a,b)
 \in {\cal I}_{n-1}$ we have when we use sufficiently large $N$
 \begin{equation}\label{h21*aa}
 \cas({G}_{n}'\cap (a,b),\aaa)<\frac{ 2^{\aaa}(b-a)^{2}}{N^{2}}(N+1)<
 (b-a)\frac{5}{N^{n+1}}.     
 \end{equation}
 This implies 
 \begin{equation}\label{*h21*bb}
 \cas({G}_{n}',\aaa)<\frac{5}{N^{n+1}}\sum_{(a,b)\in{\cal I}_{n-1}}(b-a)<\frac{5}{N^{n+1}}.
 \end{equation}
 Since the sets ${G}_{n}'$ are disjoint we obtain
 \begin{equation}\label{*h11*a}
 \cas({G}_{T,1}',\aaa)\leq \sum_{n=1}^{\oo}\cas({G}_{n}',\aaa)<\frac{5}{N}\sum_{n=1}^{\oo}
 \frac{1}{N^{n}}<\frac{1}{N}
 \end{equation}
 if $N>6$.
Next we show by mathematical induction that 
 \begin{equation}\label{*h13*a}
 \cas({G}_{T,n}',\aaa)=\sum_{(a,b)\in {\cal I}_{T,n}'} (b-a)^{\aaa}\leq \cas({G}_{T,1}',\aaa)^{n}<\frac{1}{N^{n}}.
 \end{equation}
 The case $n=1$ is \eqref{*h11*a}.
 Suppose that \eqref{*h13*a} holds for an $n\in\N$.
 By \eqref{*gvtn*} for any $(a,b)\in {\cal I}_{T,n}'$ we have
 ${G}_{T,n+1}'\cap(a,b)=\FFF_{(a,b)}^{-1}({G}_{T,1}')$ and hence
 \begin{equation}\label{*h12*a}
 \cas({G}_{T,n+1}',\aaa)\leq \sum_{(a,b)\in {\cal I}_{T,n}'} (b-a)^{\aaa}\cas({G}_{T,1}',\aaa)=
 \end{equation}
 $$\cas({G}_{T,1}',\aaa)\sum_{(a,b)\in {\cal I}_{T,n}'}(b-a)^{\aaa}\leq
 \cas({G}_{T,1}',\aaa)^{n+1}<\frac{1}{N^{n+1}}.$$
 
 \subsection{  Estimates of the size of sets $A$ on which $f$ can be convex or concave}\label{*ss37}
 Suppose $A\sse [0,1]$ and $f|_{A}$ is convex, the concave case is similar.
 For $k\in \N$ we want to estimate $\can_{N,k}(A)$. Observe that $F_{0}\sse
 {F}_{T,1}'$ and $g|_{{F}_{T,1}'}=g_{1}$ and $g_{1}|_{F_{0}}=\fff_{0}$.
Hence we have $g=\fff_{0}$ on $F_{0}$. 

For $j=1,...$  we denote by ${G}_{T,j,k}'$ 
the union of those components of ${G}_{T,j}'$ which are of length smaller than
$N^{-k}$, that is,
$${G}_{T,j,k}'=\cup\{(a,b) \in {\cal I}_{T,j}',\ |b-a|\leq {N^{-k}}  \}.$$
We put ${G}_{T,j,k}''={G}_{T,j}'\sm {G}_{T,j,k}'$, that is, ${G}_{T,j,k}''$ contains all
transitional intervals of $g_{j}$  which are of length longer than
$N^{-k}$. Its components will be denoted by $${\cai}_{T,j,k}''=\{ (a,b)\in{\cal I}_{T,j}':(a,b)\sse {G}_{T,j,k}'' \}.$$

In the next claim we estimate $\can_{N,k}(A\sm {G}_{T,1,k}'')$.
\begin{claim}\label{*nnk*}
There exists a constant $C_{g_1}>1$ not depending on $k$
 such that 
 \begin{equation}\label{*h15*a}
 {\can}_{N,k}(A\sm {G}_{T,1,k}'')<C_{g_1}\cdot N^{\aaa k} \text{ for }k=0,1,....
 \end{equation}
\end{claim}
\begin{proof}[Proof of Claim \ref{*nnk*}]
By \eqref{*h5*b} we have
\begin{equation}\label{*h15*b}
{\can}_{N,k}(F_{0}\cap(A\sm{G}_{T,1,k}''))\leq {\can}_{N,k}(F_{0})< C_{F_{0}}N^{\aaa k}.
\end{equation}

By remark \eqref{*h5*c}, $\caf_{N,k}(F_{0})$ contains all components $(a,b) \in {\cal I}_{0}$
for which $b-a\leq N^{-k}$.

Suppose that $(a,b)\in {\cal I}_{0}$ and $b-a>N^{-k}$. Select $k'<k$
 such that 
 \begin{equation}\label{*h15*c}
 N^{-k'-1}<b-a\leq N^{-k'}.
 \end{equation}
 By property \eqref{*h5*c}, $\caf_{N,k'}(F_{0})$ contains $(a,b)$ and hence by
 \eqref{*h5*b}
 \begin{equation}\label{*h16*aa}
 {\can}_{N,k'}(\cup\{ (a,b)\in{\cal I}_{0}:N^{-k'-1}<b-a\leq N^{-k'} \})\leq
 C_{F_{0}}N^{\aaa k'}.
 \end{equation}
 In later arguments we will need an upper estimate of the number of
 the intervals $(a,b)\in{\cal I}_{0}$ satisfying $N^{-k'-1}<b-a\leq N^{-k'}$.
 By \eqref{*h16*aa} these intervals can be covered by at most $C_{F_{0}}N^{\aaa k'}$ many $N^{-k'}$ grid intervals, since the length of the intervals $(a,b)$ is at least
 $N^{-k'-1}$ we obtain
 \begin{equation}\label{*h16*a}
 \# \{ (a,b)\in{\cal I}_{0}:N^{-k'-1}<b-a\leq N^{-k'} \}\leq N\cdot C_{F_{0}}N^{\aaa k'}.
 \end{equation}

 Recall that we want to estimate ${\can}_{N,k}((a,b)\cap (A\sm {G}_{T,1,k}'')).$
 We have $g_{1}(x)=g(x)$ for $x\in {F}_{T,1}'$. 
 
 For $x\in {F}_{T,1}'\cap (a,b),$ $(a,b)\in {\cal I}_{0} $ by \eqref{*h25*aa}
 \begin{equation}\label{*h25*gfe*}
 |g(x)-\fff_{1}(x)|=|g_{1}(x)-\fff_{1}(x)|<\frac{(b-a)^{2}}{N^{2}}\cdot \frac{2}{N^{2}}.
 \end{equation}
 By the definition of the functions $g_{j}(x)$ on the transitional intervals,
 from \eqref{*h8*b} and
\eqref{*h12*c} we infer
\begin{equation}\label{*h29*aa}
0< g(x)< 1 \text{ for all }x\in(0,1) \text{ and }g(0)=0,\  g(1)=1.
\end{equation}
 
 The transitional intervals $(a',b')\in{\cal I}_{1}'$ which are in $(a,b)\in{\cal I}_{0}$
 either satisfy $b'-a'> N^{-k}$ and then they do not contain any 
 point of $A\sm {G}_{T,1,k}''$, or if $b'-a'\leq N^{-k}$ then each of them can be
 covered by no more than two intervals of the form 
 $[(j-1)\cdot N^{-k},j\cdot N^{-k}]$. 
 There are at most $N+1$ many of them
and hence
\begin{equation}\label{*h17*a}
{\can}_{N,k}((a,b) \cap {G}_{1}'\cap(A\sm{G}_{T,1,k}''))<2(N+1)<3N.
\end{equation} 

If $x\in(a,b) \in {\cal I}_{0}$,  then we can refine \eqref{*h25*aa}
by using  \eqref{*h24*bb} 
\begin{equation}\label{*h17*b}
|g_{1}(x)-\fff_{1}(x)|<\frac{(b-a) ^{2\aaa^{-1}}}{N^{2\aaa^{-1}}}\cdot \frac{1}{N^{2\aaa^{-1}}}\Big (1+\frac{1}{N^{2\aaa^{-1}}}+... \Big )<
\end{equation}
$$2\frac{(b-a)^{2\aaa^{-1}}}{N^{2\cdot 2\aaa^{-1}}}\text{ if }N\text{ is sufficiently large.}$$

By \eqref{*h17*b} and by \eqref{*h21*b} in Claim \ref{*claimh28*}
\begin{equation}\label{*h21*c}
|g(x)-\fff_{1}(x)|< 6\frac{(b-a)^{2\aaa^{-1}}}{N^{4\aaa^{-1}}
}\text{ holds for }
x\in (a,b)\sm{G}_{1}'.
\end{equation}

 The transitional intervals, the components of ${G}_{1}'\cap (a,b)$
 are of length less than $2\ddd_{a,b}=2\frac{(b-a)^{2\aaa^{-1}}}{N^{2\aaa^{-1}}}$ by \eqref{*h6*d} and \eqref{*h7*a}.

For $(a',b')\in{\cal I}_{1}'$, that is, for components of ${G}_{1}'$ 
by \eqref{*h29*aa}
when $g(a')<g(b')$
we have
for $x\in (a',b')$
$$g(a')\leq g(x)\leq g(b')=g(a')+(b'-a')\text{ and }$$ $$
g(a')=\fff_{1}(a')\leq \fff_{1}(x)\leq \fff_{1}(b')=g(b')=g(a')+(b'-a')$$
and an analogous statement is valid when 
$g(a')>g(b')$.
Since by the (restricted) self-similarity property of $g$, that is by \eqref{*h12*cc}
we have
$$g(x)=g(a')+(g(b')-g(a'))\cdot g\Big (\frac{x-a'}{b'-a'}\Big )$$
we obtain
$$|g(x)-\fff_{1}(x)|\leq |b'-a'|<2\frac{(b-a)^{2\aaa^{-1}}}{N^{2\aaa^{-1}}}.$$
By \eqref{*h6*d}, $$\lll((a,b)\cap{G}_{1}')=2N \frac{(b-a)^{2\aaa^{-1}}}{N^{2\aaa^{-1}}}<2 \frac{(b-a)^{2\aaa^{-1}}}{N^{\aaa^{-1}}}$$
and hence
\begin{equation}\label{*h22*a}
\int_{(a,b)\cap {G}_{1}'}|g(t)-\fff_{1}(t)|dt< 2\frac{(b-a)^{2\aaa^{-1}}}{N^{\aaa^{-1}}}
\cdot 2\frac{(b-a)^{2\aaa^{-1}}}{N^{2\aaa^{-1}}}<
\end{equation} 
(using that by \eqref{*h5*a} and \eqref{*h7*c}, $b-a<2^{\aaa/2}N^{-1}$)
$$8\frac{(b-a)^{3\aaa^{-1}}}{N^{4\aaa^{-1}}}<16 \frac{(b-a)^{2\aaa^{-1}}}{N^{5\aaa^{-1}}}.$$

Set ${\widetilde{f}}(x)=f(a)+\int_{a}^{x}\fff_{1}(t)dt$. Then by \eqref{*h21*c}
and \eqref{*h22*a} we have for $x\in(a,b)$
\begin{equation}\label{*h22*b}
|{\widetilde{f}}(x)-f(x)|\leq \int_{a}^{x}|g(t)-\fff_{1}(t)|dt\leq
\end{equation}
$$\int_{[a,x]\sm{G}_{1}'}|g(t)-\fff_{1}(t)|dt+
\int_{[a,x]\cap{G}_{1}'}|g(t)-\fff_{1}(t)|dt\leq$$
$$6\frac{(b-a)^{2\aaa^{-1}}}{N^{4\aaa^{-1}}}(x-a)+
8\frac{(b-a)^{3\aaa^{-1}}}{N^{4\aaa^{-1}}}<
6\frac{(b-a)^{2\aaa^{-1}}}{N^{4\aaa^{-1}}}(x-a)+
16 \frac{(b-a)^{2\aaa^{-1}}}{N^{5\aaa^{-1}}}.
$$

We consider $(a,b)\in{\cal I}_{0}$ such that $b-a>N^{-k},$ and $k'<k$ satisfies \eqref{*h15*c}.
Suppose that for a $j\in \{ 1,...,N \}$
we have 
\begin{equation}\label{*h23*a}
t_{j}\in I(a,b,j).
\end{equation}
Then $t_{j}\in (a,b)\sm {G}_{1}'$ and by \eqref{*h21*c}
$$|g(t_{j})-\fff_{1}(t_{j})|<6\frac{(b-a)^{2\aaa^{-1}}}{N^{4\aaa^{-1}}}\text{ and }$$
$$\fff_{1}(t_{j})=\fff_{1}(a)+(-1)^{j+1}\frac{(b-a)^{2\aaa^{-1}}}{N^{2\aaa^{-1}}}=
g(a)+(-1)^{j+1}\frac{(b-a)^{2\aaa^{-1}}}{N^{2\aaa^{-1}}}$$
and hence for odd $j$
\begin{equation}\label{*h23*b}
g(t_{j})>g(a)+\frac{(b-a)^{2\aaa^{-1}}}{N^{2\aaa^{-1}}}-6\frac{(b-a)^{2\aaa^{-1}}}{N^{4\aaa^{-1}}}
\end{equation}
and for even $j$
\begin{equation}\label{*h23*c}
g(t_{j})<g(a)-\frac{(b-a)^{2\aaa^{-1}}}{N^{2\aaa^{-1}}}+6\frac{(b-a)^{2\aaa^{-1}}}{N^{4\aaa^{-1}}}.
\end{equation}
 
 \begin{claim}\label{*claimh31*}
Suppose that  there exists an odd $j\in \{ 1,...,N \}$
 such that 
 \begin{equation}\label{*h24*a}
 \text{one can find }x_{j}<y_{j},\ x_{j},y_{j}\in A\cap I(a,b,j).
 \end{equation}
 Then  for $j'\geq j+2$ we have $A\cap I(a,b,j')=\ess$
 and $A\cap I(a,b,j+1)$ can contain at most one element.
 \end{claim}

 \begin{proof}[Proof of Claim \ref{*claimh31*}]
 The second part of the statement of the claim is easier and we verify it first.
 Suppose that we can find $x_{j+1}<y_{j+1}$, $x_{j+1},y_{j+1}\in A\cap I(a,b,j+1).$
 Then by the Mean Value theorem   there exist $t_{j}\in (x_{j},y_{j})\sse I(a,b,j)$
 and $t_{j+1}\in(x_{j+1},y_{j+1})\sse I(a,b,j+1)$
  such that 
  \begin{equation}\label{*h24*b}
  \frac{f(y_{j})-f(x_{j})}{y_{j}-x_{j}}=f'(t_{j})=g(t_{j})\text{ and }
  \frac{f(y_{j+1})-f(x_{j+1})}{y_{j+1}-x_{j+1}}=
  f'(t_{j+1})=g(t_{j+1}).
  \end{equation}
  By \eqref{*h23*b} and \eqref{*h23*c} we obtain $g(t_{j})>g(t_{j+1})$
  but this contradicts the fact that $f$ is convex on $A$. 
  
  Suppose now that $j'\geq j+2$
  and
  \begin{equation}\label{*h25*a}
  \text{   there exists }x_{j}'\in A\cap I(a,b,j').
  \end{equation}
  Then
  \begin{equation}\label{*h25*c}
  x_{j}'-x_{j}>x_{j}'-y_{j}>\frac{(b-a)}{N}.
  \end{equation}
  By \eqref{*h23*b} and \eqref{*h24*b}
  \begin{equation}\label{*h25*b}
  \frac{f(y_{j})-f(x_{j})}{y_{j}-x_{j}}>
  g(a)+\frac{(b-a)^{2\aaa^{-1}}}{N^{2\aaa^{-1}}}-6\frac{(b-a)^{2\aaa^{-1}}}{N^{4\aaa^{-1}}}.
  \end{equation}
We also have by \eqref{*h22*b}
 \begin{equation}\label{*h26*a}
 |{\widetilde{f}}(x_{j})-{\widetilde{f}}(x_{j}')-(f(x_{j})-f(x_{j}'))|<28\frac{(b-a)^{2\aaa^{-1}}(b-a)}{N^{4\aaa^{-1}}}
 \end{equation}
 and by \eqref{*h25*c}
 \begin{equation}\label{*h26*bbb}
 \Big |\frac{{\widetilde{f}}(x_{j})-{\widetilde{f}}(x_{j}')}{x_{j}-x_{j}'}
 -\frac{f(x_{j})-f(x_{j}')}{x_{j}-x_{j}'}
 \Big |< 28 \frac{(b-a)^{2\aaa^{-1}}}{ N^{3\aaa^{-1}}}.
 \end{equation}
 The definition of $\fff_{1}$, \eqref{*h6*d}  and \eqref{*h6*b} imply that if $x_{j}'-x_{j}\leq 2(b-a) /N$
 then
 \begin{equation}\label{*h26*b}
 {\widetilde{f}}(x_{j}')-{\widetilde{f}}(x_{j})=\int_{x_{j}}^{x_{j}'}\fff_{1}(t)dt\leq \fff_{0}(a)(x_{j}'-x_{j})=
 g(a)(x_{j}'-x_{j})
 \end{equation}
 and if $x_{j}'-x_{j} >2(b-a)/N$ then
 \begin{equation}\label{*h26*bb}
 {\widetilde{f}}(x_{j}')-{\widetilde{f}}(x_{j})=\int_{x_{j}}^{x_{j}'}\fff_{1}(t)dt<
 g(a)(x_{j}'-x_{j})+\frac{(b-a)^{2\aaa^{-1}}}{N^{2\aaa^{-1}}}\cdot \frac{(b-a)}{N}.
 \end{equation}
 In both cases we obtain that
 \begin{equation}\label{*h27*b}
 \frac{{\widetilde{f}}(x_{j}')-{\widetilde{f}}(x_{j})}{x_{j}'-x_{j}}\leq g(a)+\frac{(b-a)^{2\aaa^{-1}}}{2N^{2\aaa^{-1}}}.
 \end{equation}
 Therefore, if $N>100$, \eqref{*h25*b}, \eqref{*h26*bbb} and \eqref{*h27*b} imply
 $$\frac{f(x_{j}')-f(x_{j})}{x_{j}'-x_{j}}<\frac{f(y_{j})-f(x_{j})}{y_{j}-x_{j}}$$
 and this contradicts the convexity of $f$ on $A$ and 
 proves Claim \ref{*claimh31*}.
 \end{proof}

 A similar argument can show 
 \begin{claim}\label{*claimh31b*}
Suppose that  there exists an even $j\in \{ 1,...,N \}$
 such that 
 \begin{equation}\label{*h27*c}
 \text{one can find }x_{j}<y_{j},\ x_{j},y_{j}\in A\cap I(a,b,j).
 \end{equation}
Then  for $j'\leq j-2$ we have $A\cap I(a,b,j')=\ess$
 and $A\cap I(a,b,j-1)$ can contain at most one element.
 \end{claim}

 Thus, there are at most two $j$s for which $A\cap \inte(I(a,b,j))$
 can contain more than one element.

 Suppose ${\widehat{j}}$ is  such that  $A\cap \inte(I(a,b,{\widehat{j}}))$
 contains more than one element.
 Then $\inte(I(a,b,{\widehat{j}}))={\widehat{I}}$ is a component of
 $G_{1}$ and hence it is an interval belonging to $\cai_1$.
 
 Denoting the endpoints of ${\widehat{I}}$ by ${\widehat{a}}$ and ${\widehat{b}}$ we have
 $(\widehat{a},\widehat{b})\in \cai_{1}.$ If ${\widehat{b}}-{\widehat{a}}\leq N^{-k}$
 then ${\can}_{N,k}((A\sm {G}_{T,1,k}'')\cap(\widehat{a},\widehat{b}))\leq 2.$
 
 If ${\widehat{b}}-{\widehat{a}}> N^{-k}$ then we can argue as before using $\fff_{2}$
 instead of $\fff_{1}$ to show that the part of
 $A\sm {G}_{T,1,k}''$ in transitional intervals $(a',b')\in\cai_{2}'$,
 $(a',b')\sse (\widehat{a},\widehat{b})$ can be covered by no more than $3N$
 many intervals of the form $[(j-1)\cdot N^{-k},j\cdot N^{-k}]$ and hence
 \begin{equation}\label{*h29*a}
 {\can}_{N,k}((\widehat{a},\widehat{b})\cap G_{2}'\cap (A\sm{G}_{T,1,k}''))< 3N
 \end{equation}
 and there are at most two $j$s for which $A\cap\inte(I({\widehat{a}},{\widehat{b}},j))$ can
 contain more than one element. In these intervals we need to repeat our argument, but the number of these intervals 
 can at most double at each step.
 Therefore, if we consider $(a,b)\in {\cal I}_{0}$ satisfying \eqref{*h15*c}
 then in at most $k-k'$ many steps we can obtain intervals shorter than $N^{-k}$.
 Thus with a generous upper estimate
 \begin{equation}\label{*h36*aa}
 {\can}_{N,k}((a,b)\cap (A\sm{G}_{T,1,k}''))<2^{k-k'}100\cdot N.
 \end{equation}
 By  \eqref{*h16*a} we obtain that
 $${\can}_{N,k}({G}_{0}\cap (A\sm {G}_{T,1,k}''))<\sum_{k'=1}^{k-1}N\cdot C_{F_{0}}N^{\aaa  k'}2^{k-k'}\cdot 100N= $$ 
 $$100N^{2}\cdot C_{F_{0}}N^{\aaa k}\sum_{k'=1}^{k-1}\Big (\frac{2}{N^{\aaa}}\Big )^{k-k'}< C_{F_{0}}' N^{\aaa k},$$
 where we used that we can assume that we use an $N$ so large that $2/N^{\aaa}<1$. This and \eqref{*h15*b} imply \eqref{*h15*a}.
 This concludes the proof of Claim \ref{*nnk*}.
\end{proof}
In Claim \ref{*nnk*} in \eqref{*h15*a} we
 estimated ${\can}_{N,k}(A\sm {G}_{T,1,k}'')$.
 Next we need to estimate ${\can}_{N,k}(A\cap {G}_{T,1,k}'').$
 We use
 \begin{equation}\label{*h37*a}
 {\can}_{N,k}(A\cap{G}_{T,1,k}'')\leq
 {\can}_{N,k}(A\cap{G}_{T,2,k}'')+
 {\can}_{N,k}((A\cap{G}_{T,1,k}'')\sm {G}_{T,2,k}'' )\leq
 \end{equation}
 $$\sum_{j=2}^{\oo} {\can}_{N,k}((A\cap{G}_{T,j-1,k}'')\sm {G}_{T,j,k}'' ),$$
where the last sum contains only finitely many nonzero terms since the lengths of
the longest components of $G'_{T,j,k}$ (and of $G''_{T,j,k}\sse G'_{T,j,k}$)
tend to zero as $j\to\oo$.
\begin{claim}\label{*nnkj*}
For any $j\geq 2$ we have
\begin{equation}\label{*h38*a}
{\can}_{N,k}((A\cap{G}_{T,j-1,k}'')\sm {G}_{T,j,k}'' )\leq
N^{2}\cas(G_{T,j-1}',\aaa)\cdot  C_{g_{1}} N^{\aaa k}.
\end{equation}
\end{claim}

We prove this claim later. Before doing so we finish the proof of
Theorem \ref{*5b*}.
By \eqref{*h13*a}, \eqref{*h15*a}, \eqref{*h37*a} and \eqref{*h38*a}
$${\can}_{N,k}(A)\leq C_{g_{1}}N^{\aaa k}+ \sum_{j=2}^{\oo}
 N^{2} C_{g_{1}} \frac{1}{N^{j-1}}\cdot N^{\aaa k}< C_{\aaa}
N^{\aaa k}$$
with a suitable constant $C_{\aaa}$ not depending on $k$.
This implies that $\udimm A\leq \aaa$ and ends the proof of Theorem \ref{*5b*}.
\end{proof}
 
\begin{proof}[Proof of Claim \ref{*nnkj*}]
Suppose $(a,b)\in \cai_{T,j-1,k}''$, that is,
$(a,b)$ is a component of $G_{T,j-1,k}''$.
Recall \eqref{*h12*cc},  the (restricted) self-similarity property of
$g$.
If $g$ is convex on $A\cap (a,b)$ then $g$ is also convex on
$\FFF_{(a,b)}(A\cap (a,b))$ and by \eqref{*gvtn*},
$\FFF_{(a,b)}(G'_{T,j}\cap (a,b))=G'_{T,1}$.

Suppose
\begin{equation}\label{*h39*c}
N^{-k'-1}<b-a\leq N^{-k'}\text{ for a }k'<k.
\end{equation}

If $(a',b')\in \cai_{T,j,k}''$ then $b'-a'> N^{-k}$
and $\FFF_{(a,b)}(a',b')$ is an interval
of length $\frac{b'-a'}{b-a}>\frac{N^{-k}}{b-a}$.
By \eqref{*h39*c}, $\FFF_{(a,b)}(G''_{T,j,k}\cap (a,b))$
contains all intervals of $G'_{T,1}$ which are of length at least 
$N^{-k}/(b-a)< N^{-k+k'+1}$.
Hence
$\FFF_{(a,b)}(G''_{T,j,k}\cap (a,b))\supset G''_{T,1,k-k'-1}$
and 
$$\FFF_{(a,b)}(A\cap (a,b))\sm G''_{T,1,k-k'-1}\supset \FFF_{(a,b)}((A\cap (a,b))\sm G''_{T,j,k}).$$
Thus, 
using \eqref{*h39*c} and the fact that an interval of length 
$(b-a) N^{-k+k'+1}$ can be covered by
$(N+1)$ many grid intervals of length $N^{-k}$
\begin{equation}\label{*h39*a}
N^{2} \can_{N,k-k'-1}(\FFF_{(a,b)}(A\cap (a,b))\sm G_{T,1,k-k'-1}'')\geq
{\can}_{N,k}(((a,b)\cap A)\sm G_{T,j,k}'').
\end{equation}
On the other hand, by \eqref{*h15*a} and \eqref{*h39*c}
\begin{equation}\label{*h39*b}
\can_{N,k-k'-1}(\FFF_{(a,b)}(A\cap (a,b))\sm G_{T,1,k-k'-1}'')<
C_{g_{1}}N^{\aaa(k-k'-1)}< C_{g_{1}}(b-a)^{\aaa}N^{\aaa k}.
\end{equation}
By \eqref{*h39*a}
$$
\can_{N,k}((A\cap (a,b))\sm G_{T,j,k}'')<
 N^{2} C_{g_{1}}(b-a)^{\aaa}N^{\aaa k}.
$$
Adding this for all $(a,b)\in \cai''_{T,j-1,k}$
we obtain \eqref{*h38*a}.
\end{proof} 


\end{document}